\documentclass[11pt]{amsart}
\usepackage{amscd,setspace,amssymb,amsopn,amsmath,amsthm,mathrsfs,lmodern,graphics,amsfonts,enumerate,verbatim,calc,times} 
\usepackage[all]{xy}
\usepackage[para]{threeparttable}
\usepackage[colorlinks=true, citecolor=RoyalBlue, filecolor=RoyalBlue, linkcolor=RoyalBlue,pagebackref,urlcolor = RoyalBlue, hyperindex]{hyperref}
\usepackage[dvipsnames]{xcolor}
\usepackage[referable]{threeparttablex}
\usepackage{todonotes}
\usepackage{footnote}

\usepackage{tikz-cd}
\usepackage{color, hyperref}
\usepackage[OT2,OT1]{fontenc}

\newcommand\cyr{%
\renewcommand\rmdefault{wncyr}%
\renewcommand\sfdefault{wncyss}
\renewcommand\encodingdefault{OT2}%
\normalfont
\selectfont}
\DeclareTextFontCommand{\textcyr}{\cyr}
\usepackage{amssymb,amsmath}
\DeclareFontFamily{OT1}{rsfs}{}	
\DeclareFontShape{OT1}{rsfs}{n}{it}{<-> rsfs10}{}
\DeclareMathAlphabet{\fmathscr}{OT1}{rsfs}{n}{it}
\numberwithin{equation}{section}
\newtheorem{Theoremx}{Theorem}
 
\newtheorem{theorem}{Theorem}[section]
\newtheorem*{maintheorem*}{Main Theorem}
\newtheorem{lemma}[theorem]{Lemma}

\newtheorem{proposition}[theorem]{Proposition}

\theoremstyle{definition}
\newtheorem{definition}[theorem]{Definition}

\theoremstyle{remark}

\newcommand{\Spec}{\operatorname{Spec}}

\newcommand{\Ht}{\operatorname{ht}}
\newcommand{\Height}{\operatorname{ht}}

\newcommand{\Ext}{\operatorname{Ext}}

\newcommand{\Tor}{\operatorname{Tor}}

\newcommand{\Hom}{\operatorname{Hom}}

\newcommand{\depth}{\operatorname{depth}}

\newcommand{\un}{\operatorname{un}}

\newcommand{\Q}{\mathbb{Q}}

\newcommand{\fm}{\mathfrak{m}}
\newcommand{\fp}{\mathfrak{p}}
\newcommand{\fq}{\mathfrak{q}}

\newcommand{\cP}{\mathcal{P}}

\begin{document}

\title{$F$-purity deforms in $\Q$-Gorenstein rings}

\author[Thomas Polstra]{Thomas Polstra}
\address{Department of Mathematics, University of Alabama, Tuscaloosa, AL 35487}
\thanks{Polstra was supported in part by NSF Grant DMS \#2101890 and by a grant from the Simons Foundation, Grant \#814268, MSRI.}
\email{tmpolstra@ua.edu}
\urladdr{\url{https://thomaspolstra.github.io/}}

\author[Austyn Simpson]{Austyn Simpson}
\thanks{Simpson was supported by NSF Grant DMS \#2202890.}
\address{Department of Mathematics, University of Michigan, Ann Arbor, MI 48109}
\email{austyn@umich.edu}

\begin{abstract}
We show that $F$-purity deforms in local $\Q$-Gorenstein rings of prime characteristic $p>0$. Furthermore, we show that $F$-purity is $\fm$-adically stable in local Cohen-Macaulay $\Q$-Gorenstein rings.
\end{abstract}

\maketitle

\section{Introduction}
Let $(R,\fm,k)$ be a Noetherian local ring of prime characteristic $p>0$, and let $F:R\rightarrow R$ be the Frobenius endomorphism $r\mapsto r^p$. This article is concerned with \emph{$F$-pure} rings, i.e. rings for which the Frobenius is a pure morphism. The notion of $F$-purity is a central pillar in the vast and expanding web of Frobenius singularities emerging from the celebrated works of Hochster-Roberts and Hochster-Huneke on rings of invariants \cite{HR74,HR76} and tight closure \cite{HH90}, respectively. Some of these so-called $F$-singularities include $F$-regular, $F$-rational, and $F$-injective which all are defined (with the foundational work of Kunz \cite{Kun69} as motivation) by some weakening of the flatness of Frobenius. Although fascinating from a purely commutative algebraic viewpoint, these classes of mild singularities have garnered widespread attention in the past few decades due to their connections to singularities in complex birational geometry \cite{BST17,HW02,Sch09b,Smi97,ST17} and (recently) to those in mixed characteristic \cite{MS21,MSTWW22,BMPSTWW20}.

A natural question to consider of any class of singularities is whether they deform. Recall that a property $\cP$ of local rings is said to \emph{deform} if given a local ring $(R,\fm,k)$ and a non-zero-divisor $f\in\fm$ such that $R/(f)$ is $\cP$, $R$ is $\cP$ too. Deformation holds for many familiar local properties\footnote{see \cite[Table 2]{Mur22} for a summary} and can be viewed as a weak version of \emph{inversion of adjunction}.

The study of this phenomenon in the realm of $F$-singularities has a long history dating back to \cite{Fed83} wherein Fedder initiated the exploration of whether $F$-injectivity deforms, a question which remains open as of the present article. It was known early on that all four $F$-singularities deform if the ambient ring is Gorenstein \cite{Fed83,HH94}, but outside of this setting the situation is much more delicate (see Table \ref{summary-table} for a summary). The primary subtlety is that if one does not impose $\Q$-Gorenstein assumptions on the ambient ring, then there are counterexamples (due to Fedder \cite{Fed83} and Singh \cite{Sin99}) which show that $F$-purity (resp. $F$-regularity) do not deform. This obstruction has a counterpart in the characteristic zero minimal model program, namely that log canonical and log terminal singularities (the supposed analogues of $F$-purity and $F$-regularity respectively, evidenced by \cite{Har98, HW02, Her16, Tak13}) do not deform without additional $\Q$-Gorenstein hypotheses (see \cite[Example 9.1.7 \& Remark 9.1.15]{Ish18}). The main contribution of the present article is a complete resolution to the question of whether $F$-purity deforms in $F$-finite rings, which we answer affirmatively in the $\Q$-Gorenstein scenario.

\begin{Theoremx} (Theorem \ref{Theorem F-purity deforms})
\label{Main theorem F-purity}
Let $(R,\fm,k)$ be a local $F$-finite $\Q$-Gorenstein ring of prime characteristic $p>0$. Suppose that $f\in \fm$ is a non-zero-divisor such that $R/(f)$ is Gorenstein in codimension $1$, $(S_2)$, and $F$-pure. Then $R$ is $F$-pure.
\end{Theoremx}

The analogous statement regarding the deformation of $F$-regularity was obtained in \cite{AKM98} over two decades ago via tight closure techniques which by present standards are considered classical. Experts were aware that special cases of Theorem \ref{Main theorem F-purity} were achievable if the $\Q$-Gorenstein index of $R$ is coprime to the characteristic (see \cite[Remark 4.10(1)]{HW02} and \cite[Proposition 7.2]{Sch09a}), but the index-free version of the conjecture remained elusive until now.

We wish to emphasize that the history of this problem mirrors that of the deformation and inversion of adjunction problems for log terminal and log canonical singularities in the characteristic zero MMP in two ways: the failure already described above outside of the $\Q$-Gorenstein scenario, and in the jump in difficulty between the two methods. Log terminal singularities were seemingly known to satisfy inversion of adjunction in the program's infancy (see e.g. \cite[Theorem 5.50]{KM98}). By contrast, the analogous conjecture for log canonicity vexed researchers for several decades before Kawakita's solution in \cite{Kaw07}. The parallels that one may draw in the prime characteristic world thus unsurprisingly point to a higher degree of subtlety in the deformation problem for $F$-purity than that of $F$-regularity, which we believe is visible from our methods in Section \ref{Section F-purity deforms and is stable}. We employ a novel strategy involving cyclic covers to remove $p$-torsion of the canonical module, and we remark that this trick cannot be used to obtain a more general inversion of adjunction statement for $F$-pure pairs.

Shifting gears, a related problem that we conclude the article with is that of $\fm$-adic stability for $F$-purity. Following \cite{DSS20}, a property $\cP$ of local rings is said to be \emph{$\fm$-adically stable} if given a local ring $(R,\fm,k)$ and a non-zero-divisor $f\in \fm$ such that $R/(f)$ is $\cP$, there exists an integer $N>0$ so that for every $\epsilon\in\fm^N$, $R/(f+\epsilon)$ is also $\cP$. This notion of stability is related to both deformation and to finite determinacy, has early roots in \cite{Eis74,Hir65,Sam56}, and has sparked a flurry of recent activity \cite{DSS20,Dua22,MQS20,PS18,QT21}. Provided that $\cP$ descends along faithfully flat maps and that $\cP$ passes from $R$ to $R[x]_{(\fm,x)}$ then stability of $\cP$ implies deformation of $\cP$ \cite[Theorem 2.4]{DSS20}.  \emph{Op. cit.} initiates the study of $\fm$-adic stability for $F$-singularities, and it is shown that stability holds (or fails) in many of the same instances that deformation does (see Table \ref{summary-table}). We bolster this program by showing that $F$-purity is $\fm$-adically stable in $\Q$-Gorenstein \emph{Cohen-Macaulay} rings.

\begin{Theoremx} (Theorem \ref{Theorem F-purity deforms Q-Gorenstein index p^e_0})
Let $(R,\fm,k)$ be a local $F$-finite Cohen-Macaulay $\Q$-Gorenstein ring of prime characteristic $p>0$. Suppose that $f\in \fm$ is a non-zero-divisor such that $R/(f)$ is Gorenstein in codimension $1$, and $F$-pure. Then there exists $N>0$ so that $R/(f+\epsilon)$ is $F$-pure for every $\epsilon\in\fm^N$.
\end{Theoremx}

\subsection{Notation, Conventions, and organization of the paper}

All rings in this article are commutative with unit and Noetherian. Typically, our interests lie in the development of the theory of prime characteristic rings, especially in Section~\ref{Section F-purity deforms and is stable}. If $R$ is of prime characteristic $p>0$, then we denote by $F:R\to R$ the Frobenius endomorphism. If $I\subseteq R$ is an ideal, for each $e\in\mathbb{N}$ we denote by $I^{[p^e]}$ the ideal $\langle r^{p^e}\mid r\in I\rangle$. For each $e\in\mathbb{N}$ and $R$-module $M$ we let $F^e_*M$ denote the $R$-module obtained from $M$ via restriction of scalars along $F^e$. That is, $F^e_*M$ agrees with $M$ as an abelian group, and given $r\in R$ and $m\in M$ we have $r\cdot F^e_*(m)=F^e_*(r^{p^e} m)$ where $F^e_*(m)$ is the element of $F^e_*M$ corresponding to $m$. Often times we assume that $R$ is \emph{$F$-finite}, i.e. if $M$ is a finitely generated $R$-module, so too is $F^e_* M$ for each $e\in\mathbb{N}$.

We often assume that a local ring is the homomorphic image of a regular local ring, an assumption that is satisfied by every $F$-finite local ring, \cite[Remark~13.6]{Gab04}. If $R\cong S/I$ where $S$ is a regular local ring then $\Ext^{\Height(I)}_S(R,S)$ is a choice of canonical module of $R$. We caution the reader that we do \emph{not} require $\Q$-Gorenstein rings to be Cohen-Macaulay or normal; see Definition \ref{Definition Q-gor} for a precise definition.  

In the study of $\fm$-adic stability, we occasionally use the shorthand ``$R/(f+\epsilon)$ satisfies property $\cP$ for all $\epsilon\in\fm^{N\gg 0}$" to mean that there exists an integer $N>0$ so that $R/(f+\epsilon)$ is $\cP$ for all $\epsilon\in\fm^N$.

\section{Preliminary Results}
\label{Section Preliminary}

\subsection{Generalized divisors, divisorial ideals, and cyclic covers}

A ring $R$ is $(G_1)$ if $R$ is Gorenstein in codimension $1$. Suppose that $(R,\fm,k)$ is a local $(G_1)$ ring satisfying Serre's condition $(S_2)$, and let $L$ denote the total ring of fractions of $R$. Following \cite{Har94} there is a well-defined notion of linear equivalence  on the collection of divisors which are Cartier in codimension $1$, and hence there is a well-defined additive structure on such divisors up to linear equivalence. If $D$ is a Weil divisor which is Cartier in codimension $1$ then we let $R(D)$ denote the corresponding divisorial ideal of $R$, i.e.
\[
R(D)=\{x\in L\setminus\{0\}\mid \mbox{div}(x)+D\geq 0\}\cup\{0\}.
\]
In other words, $R(D)$ denotes the global sections of $\mathcal{O}_{\Spec R}(D)$. We assume that $R$ admits a dualizing complex and thus admits a canonical module $\omega_R$. Under our assumptions we have that $\omega_R\cong R(K_X)$ for a choice of \emph{canonical divisor} $K_X$ on $X=\Spec(R)$ which is Cartier in codimension $1$ (see \cite[Proposition 2.8]{Har94}).

\begin{definition}\label{Definition Q-gor}
 A ring $R$ is said to be \emph{$\Q$-Gorenstein} if $R$ has a canonical module and
\begin{enumerate}
    \item $R$ is $(G_1)$ and $(S_2)$ with choice of canonical divisor $K_X$;
    \item there exists an integer $n$ so that $nK_X$ is Cartier.
\end{enumerate}  The least integer $n$ so that $nK_X$ is Cartier is referred to as the \emph{$\Q$-Gorenstein index of $R$}. The $\Q$-Gorenstein index is independent of choice of canonical divisor as $K_X$ is Cartier in codimension $1$, i.e. $R$ is $(G_1)$.
\end{definition}

If $D\leq  0$ is an anti-effective divisor, Cartier in codimension $1$, then $I=R(D)\subseteq R$ is an ideal of pure height $1$. Conversely, if $I\subseteq R$ is an ideal of pure height $1$, principal in codimension $1$, then $I=R(D)$ for some anti-effective divisor $D$ which is Cartier in codimension $1$. For every natural number $N\geq 1$ the divisorial ideal $R(ND)$ agrees with $I^{(N)}$, the $N$th symbolic power of $I$. 

Suppose that $D$ is a $\Q$-Cartier (i.e. torsion, since $R$ is local) divisor of index $n$ which is Cartier in codimension $1$. Suppose that $R(nD)=R\cdot u$. The cyclic cover of $R$ corresponding to $D$ is the finite $R$-algebra
\[
R\to R_D:= \bigoplus_{i=0}^{n-1}R(-iD)t^{-i}\cong \bigoplus_{i=0}^{\infty}R(-iD)t^{-i}/(u^{-1}t^{-n}-1).
\]
The map $R\to R_D$ is a finite $R$-module homomorphism and the ring $R_D$ decomposes into a direct sum of $(S_2)$ $R$-modules, hence $R_D$ is $(S_2)$. Furthermore, we can explicitly describe the canonical module of $R_D$ as 
\[
\omega_{R_D}\cong \Hom_R(R_D, R(K_X))\cong \bigoplus_{i=0}^{n-1}R(K_X+iD)t^i.
\]
The above computation commutes with localization and hence $R_D$ is Gorenstein in codimension $1$. Indeed, if $P$ is a height $1$ prime ideal of $R_D$ then $\fp=P \cap R$ is a height $1$ prime of $R$ and $R_D\otimes_R R_\fp \cong \omega_{R_D}\otimes_R R_\fp$. Localizing further we find that $(\omega_{R_D})_{P}\cong (R_D)_P$. In what follows, we denote $(-)^\vee:=\Hom_{R_D}(-,R_D)$.

The following lemma is well-known to experts. We sketch a proof for sake of completeness. 
\begin{lemma}
\label{Lemma cyclic cover changing index}
Let $(R,\fm,k)$ be a local $\Q$-Gorenstein ring with canonical divisor $K_X$ on $X=\Spec(R)$. Suppose that $R$ is of $\Q$-Gorenstein composite index $m\cdot n$ and $D:=mK_X$. Then the cyclic cover $R_D$ corresponding to the $\Q$-Cartier divisor $D$ is $\Q$-Gorenstein of index $m$.
\end{lemma}

\begin{proof}
We identify $\omega_{R_D}$ as $\bigoplus_{i=0}^{n-1}R(K_X+iD)t^i$ as above. Notice that after relabeling indices we have
\[
\omega_{R_D}\cong \bigoplus_{i=0}^{n-1}R(K_X+iD)t^i\cong \bigoplus_{i=0}^{n-1}R(K_X-(m-i)D)t^{-(m-i)}\cong (R(K_X)R_D)^{\vee\vee}
\]
where $R(K_X)R_D$ is the extension of $R(K_X)$ to the cyclic cover $R_D$ and $(R(K_X)R_D)^{\vee \vee}$ is the reflexification of $R(K_X)R_D$. Therefore if $K_Y$ is a choice of canonical divisor of $Y=\Spec(R_D)$ and $\pi: Y\to X$ is the associated map of affine schemes then $K_Y\sim \pi^*K_X$, where $\pi^*K_X$ is the divisor on $Y$ such that $R_D(\pi^*K_X)=(R(K_X)R_D)^{\vee\vee}$. Therefore $mK_Y\sim \pi^*mK_X=\pi^*D$, and $\pi^*D\sim 0$ as $(R(D)R_D)^{\vee\vee}$ is isomorphic to a shift of $R_D$.

If $0<m'<m$ then $m'K_Y\sim\pi^*m'K_X$ is not principal. Else, $R_D$ is isomorphic to a shift of $(R(m'K_Y)R_D)^{\vee\vee}$. If this is the case, then $(m'-im)K_X\sim 0$ for some $0\leq i<n$. This is a contradiction as $1\leq -(m'-im)<mn$ and we assumed the index of $K_X$ to be $mn$.
\end{proof}

In Section~\ref{Section F-purity deforms and is stable} we will need to understand the behavior of divisorial ideals under base change. In particular, we will need to know that the base change of a divisorial ideal remains a divisorial ideal under suitable hypotheses. The following lemma will be a key tool in accomplishing this.

\begin{lemma}
\label{Lemma base change of divisorial ideal}
Let $(R,\fm,k)$ be a local ring which is $(S_2)$ and $(G_1)$. Suppose that $J\subsetneq R$ is a canonical ideal of $R$. Let $f\in \fm$ be a non-zero-divisor so that $f$ is regular on $R/J$ and $R/(f)$ is $(S_2)$ and $(G_1)$. If for each $\fp\in V(f)$ we have that $\depth(J^{(n)}R_\fp)\geq \min \{\Ht(\fp),3\}$, (i.e. $J^{(n)}R_\fp$ is $(S_3)$ for each $\fp\in V(f)$), then $J^{(n)}/fJ^{(n)}\cong \left((J,f)/(f)\right)^{(n)}$. In particular, the isomorphism holds when $J^{(n)}$ is Cohen-Macaulay.
\end{lemma}

\begin{proof}
Consider the short exact sequence
\[
0\to J^{(n)}\to R\to R/J^{(n)}\to 0.
\]
Since $f$ is regular on $R$ and $R/J$ we have that $f$ is regular on $R/J^{(n)}$ for all $n$ and $\Tor_1(R/(f),R/J^{(n)})\cong H_1(f;R/J^{(n)})=0$, hence there is a short exact sequence
\[
0\to \frac{J^{(n)}}{fJ^{(n)}}\to \frac{R}{(f)}\to \frac{R}{(J^{(n)},f)}\to 0,
\]
and therefore $J^{(n)}/fJ^{(n)}\cong (J^{(n)},f)/(f)$.

We notice that the ideals $(J^{(n)},f)/(f)$ and $\left((J,f)/(f)\right)^{(n)}$ agree at codimension $1$ points of $\Spec(R/(f))$ by the assumption that $R/(f)$ is $(G_1)$. Moreover, our assumptions imply that $(J^{(n)},f)/(f)$ is an $(S_2)$ $R/(f)$-module, which together with the previous sentence tells us that $(J^{(n)},f)/(f)\cong \left((J,f)/(f)\right)^{(n)}$ as claimed.
\end{proof}

\subsection{\texorpdfstring{$\mathbb{Q}$}{ℚ}-Gorenstein rings and the \texorpdfstring{$(G_1)$}{G\_1} property}

The following result shows that being Gorenstein in codimension $1$ deforms and is $\fm$-adically stable. The proof uses a standard trick involving the Krull intersection theorem which we use elsewhere in this article with details omitted.

\begin{proposition}
\label{Proposition Gorenstein in codimension 1 is stable}
Let $(R,\fm,k)$ be a local equidimensional and catenary ring which admits a canonical module. Suppose that $\dim R\geq 2$. If $f\in \fm$ is a non-zero-divisor and $R/(f)$ is $(G_1)$, then $R$ is $(G_1)$. If $R$ is further assumed to be excellent, then the rings $R/(f+\epsilon)$ are $(G_1)$ for all $\epsilon\in\fm^{N\gg 0}$.
\end{proposition}

\begin{proof}
Suppose that $\fp$ is a height $1$ prime. Then the ideal $(\fp,f)$ has height no more than $2$ since we are assuming $R$ is equidimensional and catenary and hence we can choose a height $2$ prime $\fq\in V((\fp,f))$. Then $R_\fq/(f)R_\fq$ is Gorenstein by assumption. The property of being Gorenstein deforms \cite[\href{https://stacks.math.columbia.edu/tag/0BJJ}{Tag 0BJJ}]{stacks-project} and therefore $R_\fq$ is Gorenstein. Localizing further we find that $R_\fp$ is Gorenstein as well, hence $R$ is $(G_1)$.

The non-Gorenstein locus of an excellent local ring is a closed subset (e.g. by \cite[Theorem 24.6]{Mat86}). Therefore since the ring $R$ is Gorenstein in codimension $1$, the height $2$ primes in the non-Gorenstein locus of $R$ form a finite set, say $\{\fp_1,\dots, \fp_t\}$. Furthermore, $f\not\in\fp_i$ for all $i$ by the assumption that $R/(f)$ is $(G_1)$. Moreover, $\bigcap_N(\fm^N+\fp_i)=\fp_i$ for all $i$ by Krull's intersection theorem, so there exists $N_i> 0$ so that $f\not\in\fm^{N_i}+\fp_i$. Choosing $N\geq \max_i\{N_i\}$ we obtain that $f+\epsilon\not\in\fp_i$ for all $i$ and all $\epsilon\in\fm^N$. Hence, $R/(f+\epsilon)$ is $(G_1)$ for all $\epsilon\in\fm^N$. 
\end{proof}

\begin{lemma}
\label{lemma base change the canonical module}
Let $(R,\fm,k)$ be a local ring and $f\in \fm$ a non-zero-divisor of $R$ such that $R/(f)$ is $(G_1)$ and $(S_2)$. Further suppose that $R$ has a canonical module. Let $J\subsetneq R$ be a canonical ideal of $R$ such that $f$ is regular on $R/J$. If $\left((J,f)/(f)\right)^{\un}$ denotes the intersection of the minimal primary components of the ideal $(J,f)/(f)$ of $R/(f)$ then $\left((J,f)/(f)\right)^{\un}\cong \omega_{R/(f)}$.
\end{lemma}

\begin{proof}
 Consider the short exact sequence
\[
0\to R\xrightarrow{\cdot f}R\to R/(f)\to 0.
\]
Observe that $\omega_{R/(f)}\cong \Ext^1_R(R/(f),J)$ by \cite[Satz 5.12]{HZ71} and so there is an exact sequence
\[
0\to J\xrightarrow{\cdot f}J\to \omega_{R/(f)}\to \Ext^1_S(R,J)
\]
and hence there is a left exact sequence
\[
0\to \frac{J}{fJ}\cong \frac{(J,f)}{(f)}\to \omega_{R/(f)}\to \Ext_S^{1}(R,J).
\]
By the assumption that $R/(f)$ is $(S_2)$, we know that $\Ext_S^{1}(R,S)$ is not supported at any height $1$ component of $R/(f)$. Hence $(J,f)/(f)\to \omega_{R/(f)}$ is an isomorphism at codimension $1$ points of $R/(f)$. By \cite[Proposition~1.11]{Har94} we have that if $(-)^\vee$ denotes $\Hom_R(-,R/(f))$ then
\[
\left(\frac{(J,f)}{(f)}\right)^{\vee\vee}\cong \omega_{R/(f)}.
\]
Therefore the lemma is proven as $\left((J,f)/(f)\right)^{\vee\vee}\cong \left((J,f)/(f)\right)^{\un}$.
\end{proof}

\begin{proposition}
\label{Proposition Q-Gorenstein index under base change}
Let $(R,\fm,k)$ be an excellent equidimensional local $\Q$-Gorenstein ring of index $n$ satisfying Serre's condition $(S_2)$. Suppose that $f\in \fm$ is a non-zero-divisor and $R/(f)$ is $(G_1)$ and $(S_2)$. Then $R/(f)$ is $\Q$-Gorenstein of index dividing $n$. Moreover, we may choose canonical ideal $J\subsetneq R$ and element $a\in J$ so that
\begin{enumerate}
    \item $J^{(n)}=(a)$ and $J/fJ\cong (J,f)/(f)$;
    \item $\left((J,f)/(f)\right)^{\un}\cong \omega_{R/(f)}$ and $\left(((J,f)/(f)\right)^{\un})^{(n)}=(a,f)/(f)$.
\end{enumerate}
\end{proposition}

\begin{proof}
By Proposition~\ref{Proposition Gorenstein in codimension 1 is stable} the ring $R$ is $(G_1)$. We first show the existence of a canonical ideal $J\subseteq R$ so that $((J,f)/(f))^{\un}\subseteq R/(f)$ is a canonical ideal of $R/(f)$. We will then show that $(((J,f)/(f))^{\un})^{(n)}$ is a principal ideal of $R/(f)$, i.e. the $\Q$-Gorenstein index of $R/(f)$ divides the $\Q$-Gorenstein index of $R$ as claimed.

Suppose that $\omega_R$ is a canonical module of $R$. Let $W$ be the complement of the union of the height $1$ components of $f$. Then $(\omega_R)_W\cong R_W$. Thus there exists $u\in \omega_R$ and ideal $J\subseteq R$ not contained in any height $1$ component of $(f)$, so that $\omega_R\cong J\cdot u$. Then $J\subseteq R$ is a canonical ideal of $R$. Moreover, since the components of $J$ are disjoint from the components of $(f)$ we may assume that $J$ is an ideal of pure height $1$ and $f$ is a regular element of $R/J$. As in the proof of Lemma~\ref{Lemma base change of divisorial ideal} we have that $J/fJ\cong (J,f)/(f)$ and we have by Lemma~\ref{lemma base change the canonical module} that $((J,f)/(f))^{\un}\cong \omega_{R/(f)}$.

Suppose that $J^{(n)}=(a)$. We claim that $(\left((J,f)/(f)\right)^{\un})^{(n)}=(a,f)/(f)$. Equivalently, we need to show that if $\fp\in V(f)$ with $\Ht\fp=2$ then $\left((J^n,f)/(f)\right)R_\fp=\left( (a,f)/(f)\right) R_\fp$. We are assuming $R/(f)$ is $(G_1)$ and hence $R_\fp/(f)R_\fp$ is Gorenstein. The property of being Gorenstein deforms and therefore $R_\fp$ is Gorenstein. In particular, $J^NR_\fp=J^{(N)}R_\fp$ for every $N\in \mathbb{N}$ and so
\[
\left((J^n,f)/(f)\right)R_\fp=\left((J^{(n)},f)/(f)\right)R_\fp=\left( (a,f)/(f)\right) R_\fp
\]
as claimed. Therefore the index of $R/(f)$ divides the index of $R$.
\end{proof}

\subsection{Degeneracy ideals} 

Our study of $\fm$-adic stability of $F$-purity (Theorem \ref{Theorem F-purity deforms Q-Gorenstein index p^e_0}) requires a careful analysis of the Frobenius degeneracy ideals of $R$ which is partially contained in this subsection. We suppose that $(R,\fm,k)$ is a local $F$-finite ring of prime characteristic $p>0$. The $e$th Frobenius degeneracy ideal of $R$ is the ideal
\[
I_e(R):=\{c\in R\mid R\xrightarrow{1\mapsto F^e_*c}F^e_*R\mbox{ is not pure}\}.
\]
Frobenius degeneracy ideals were introduced by Yao in \cite{Yao06} and by Aberbach and Enescu in \cite{AE05} and play a prominent role in prime characteristic commutative algebra, especially in the study of $F$-regularity and $F$-purity. It is important for us to notice that a ring $R$ is $F$-pure if and only if $I_e(R)$ is a proper ideal for some (equivalently every) natural number $e\in\mathbb{N}$. 

If $R$ is Cohen-Macaulay (e.g if $R$ is $F$-regular) the Frobenius degeneracy ideals $I_e(R)$ may be realized as certain colon ideals, and this viewpoint has been central to the study of the ``weak implies strong" conjecture and $F$-signature theory (see \cite{AP19, HL02, PS18, PT18, WY04} for example).

\begin{proposition}[{\cite[Lemma~6.2]{PT18}}]
\label{Proposition degeneracy ideals as colon ideals}
Let $(R,\fm,k)$ be an $F$-finite local Cohen-Macaulay ring of prime characteristic $p>0$. Suppose that $R$ admits a canonical ideal $J\subsetneq R$, $0\not= x_1\in J$ a non-zero-divisor, $x_2,\ldots, x_d$ parameters on $R/(x_1)$, and suppose that $u$ generates the socle of the $0$-dimensional Gorenstein quotient $R/(J,x_2,\ldots,x_d)$. Then for each $e\in \mathbb{N}$ there exists $t_e\in\mathbb{N}$ so that for all $t\geq t_e$
\[
I_e(R)=(x_1^{t-1}J,x_2^t,\ldots,x_d^t)^{[p^e]}:_R (x_1\cdots x_d)^{(t-1)p^e}u^{p^e}.
\]
\end{proposition}
\begin{proof}
Under the listed assumptions, it follows that $R$ is approximately Gorenstein by \cite[page 53]{AL03} (see definition in \textit{op. cit.}). The statement then follows immediately from \cite[Lemma~6.2]{PT18}.
\end{proof}

The study of Frobenius splittings often requires carefully manipulating the colon ideals described above. We first discuss a powerful technique at our disposal in this vein when the parameter element $x_2$ is chosen so that $x_2$ multiplies the canonical ideal $J$ into a principal ideal contained in $J$. Such elements can be found provided $R$ is $(G_1)$. Choosing such a parameter element conveniently allows us to switch freely between bracket powers of the canonical ideal with symbolic powers thereof, which we record in more detail in the following lemma.

\begin{lemma}
\label{Lemma colon ideal in G_1 rings}
Let $(R,\fm,k)$ be a local $F$-finite Cohen-Macaulay $(G_1)$ ring of prime characteristic $p>0$ and of Krull dimension $d$ at least $2$. Suppose that $J\subsetneq R$ is a choice of canonical ideal of $R$ and $x_1\in J$ is a non-zero-divisor. Then there exists a parameter element $x_2$ on $R/(x_1)$ and element $0\not= a\in J$ such that $x_2J\subseteq (a)$. Moreover, for any choice of parameters $y_3,\ldots, y_d$ on $R/(x_1,x_2)$ and $e,N\in \mathbb{N}$ we have that
\begin{align*}
(J^{[p^e]},x_2^{Np^e},y_3,\ldots, y_d):_Rx_2^{(N-1)p^e}&=(J^{(p^e)},x_2^{Np^e},y_3,\ldots, y_d):_Rx_2^{(N-1)p^e}\\
& = (J^{(p^e)},x_2^{2p^e},y_3,\ldots, y_d):_Rx_2^{p^e}\\
&= (J^{[p^e]},x_2^{2p^e},y_3,\ldots, y_d):_Rx_2^{p^e}.
\end{align*}
\end{lemma}
\begin{proof}
We refer the reader to \cite[Lemma~6.7(i)]{PT18} where the first named author and Tucker record a proof of this lemma under the additional assumption that $R$ is a normal domain. We observe that the normality assumption is not necessary and one only needs that $J$ is principal in codimension $1$ for the methodology of \cite{PT18} to be applicable.
\end{proof}

Another tactic that we employ in the study of the colon ideals appearing in Proposition~\ref{Proposition degeneracy ideals as colon ideals} allows us to remove the $x_1$ term. The following lemma is well-known to experts, but we record a detailed proof for the reader's convenience.

\begin{lemma}
\label{Lemma Removing the x1}
Let $(R,\fm,k)$ be a local $F$-finite Cohen-Macaulay ring of prime characteristic $p>0$ and of Krull dimension $d$. Suppose that $J\subsetneq R$ is a choice of canonical ideal of $R$ and $x_1\in J$ is a non-zero-divisor. Then for any choice of parameters $y_2,\ldots,y_d$ of $R/(x_1)$ we have that
\[
((x_1^{t-1}J)^{[p^e]},y_2,\ldots,y_d):_R x_1^{(t-1)p^e}= (J^{[p^e]},y_2,\ldots,y_d).
\]
\end{lemma}

\begin{proof}
Clearly the ideal on the right-hand side of the claimed equality is contained in the left-hand side. Suppose that $r\in ((x_1^{t-1}J)^{[p^e]},y_2,\ldots,y_d):_R x_1^{(t-1)p^e}$, i.e.
\[
rx_1^{(t-1)p^e}\in ((x_1^{t-1}J)^{[p^e]},y_2,\ldots,y_d).
\]
Equivalently, there exists an element $j\in J^{[p^e]}$ such that
\[
(r-j)x_1^{(t-1)p^e}\in (y_2,\ldots,y_d).
\]
The element $x_1^{(t-1)p^e}$ is a non-zero-divisor on $R/(y_2,\ldots,y_d)$ and therefore
\[
r-j\in (y_2,\ldots,y_d)
\]
and therefore $r\in (J^{[p^e]},y_2,\ldots,y_d)$ as claimed.
\end{proof}

\section{Deformation and stability of \texorpdfstring{$F$}{F}-purity}
\label{Section F-purity deforms and is stable}

\subsection{\texorpdfstring{$F$}{F}-purity deforms in \texorpdfstring{$\Q$}{ℚ}-Gorenstein rings}

Let us discuss our strategy to resolving the deformation of $F$-purity problem. Suppose that $(R,\fm,k)$ is a $\Q$-Gorenstein $F$-finite local ring of prime characteristic $p>0$ and $f\in \fm$ is a non-zero-divisor such that $R/(f)$ is $(G_1)$, $(S_2)$, and $F$-pure. Suppose that $K_X$ is a choice of canonical divisor on $X=\Spec(R)$, $np^eK_X\sim 0$, $n$ is relatively prime to $p$, and $R(np^eK_X)=R\cdot u$. We let $D$ be the divisor $nK_X$ and $S=\bigoplus_{i=0}^{\infty }R(-iD)t^{-i}/(u^{-1}t^{-(p^e-1)}-1)$ be the cyclic cover of $R$ corresponding to the divisor $D$. Then the ring $R$ is a direct summand of $S$ and therefore if $S$ is $F$-pure then $R$ is $F$-pure. The ring $S$ is $\Q$-Gorenstein of index relatively prime to $p$, see Lemma~\ref{Lemma cyclic cover changing index}. Thus, if we are able to show $S/fS$ is $F$-pure, then we have reduced solving the deformation of $F$-purity problem to the scenario that the $\Q$-Gorenstein index is relatively prime to the characteristic. As discussed in the introduction, deformation of $F$-purity in this scenario was well-understood by experts before being recorded in the literature by Schwede, \cite[Proposition~7.2]{Sch09a}. Our strategy to show that $S/fS$ is $F$-pure is to show that $S/fS$ is a cyclic cover of $R/(f)$ and then utilize \cite[Proposition~4.20]{Car22} to conclude that $S/fS$ is $F$-pure.\footnote{Carvajal-Rojas makes the running assumption in \cite{Car22} that all rings are essentially of finite type over an algebraically closed field. Furthermore, it is assumed in \cite[Proposition~4.20]{Car22} that $R$ is normal. However, the proof of \cite[Proposition~4.20]{Car22} works verbatim under the milder assumptions that $(R,\fm,k)$ is any $F$-finite local ring which is $(G_1)$, $(S_2)$, and $D$ is a divisor which is Cartier in codimension $1$. We also refer the reader to \cite[Chapter~5]{MP21} for an independent treatment of these results.} 

For the sake of convenience, we record a proof that $F$-purity deforms in $\Q$-Gorenstein rings whose index is relatively prime to the characteristic. We do this since the reader might wish to avoid the extra technicalities of \cite{Sch09a} where the deformation of $F$-purity along a Weil divisor $D$ of a pair $(R,\Delta)$ is considered.

\begin{proposition}
\label{Proposition F-purity deforms if index not divisible by p}
Let $(R,\fm,k)$ be a local $F$-finite ring of prime characteristic $p>0$. Suppose that $R$ is $\Q$-Gorenstein of index relatively prime to $p$ and $f\in \fm$ is a non-zero-divisor such that $R/(f)$ is $(G_1)$, $(S_2)$, and $F$-pure. Then $R$ is $F$-pure.
\end{proposition}

\begin{proof}
Because the index of $K_X$ is relatively prime to $p$ there exists an $e$ so that $(1-p^e)K_X\sim 0$. To show that $R$ is $F$-pure we will show that every $R$-linear map $F^e_*R/(f)\xrightarrow{\varphi} R/(f)$ can be lifted to a map $\Phi:F^e_*R\to R$. That is there exists $\Phi: F^e_*R\to R$ so that the following diagram is commutative:
\begin{equation*}
    \begin{tikzcd}
    F^e_*R\arrow[r,dashed,"\Phi"]\arrow[d] & R\arrow[d] \\
    F^e_*R/(f)\arrow[r,"\varphi"] & R/(f).
    \end{tikzcd}
\end{equation*}
To do this we will consider a canonical map $\Psi:\Hom_R(F^e_*R,R)\to \Hom_{R/(f)}(F^e_*R/(f),R/(f))$, described below, and show that $\Psi$ is onto.

The map $\Psi$ factors as $\Psi_2\circ \Psi_1$ where 
\[
\Psi_1:\Hom_R(F^e_*R,R)\to \Hom_R(F^e_*R,R/(f))\cong \Hom_{R/(f)}(F^e_*R/(f^{p^e}),R/(f))
\]
is the map obtained by applying $\Hom_R(F^e_*R,-)$ to $R\to R/(f)$ and 
\[
\Psi_2:\Hom_{R/(f)}(F^e_*R/(f^{p^e}),R/(f))\to \Hom_{R/(f)}(F^e_*R/(f),R/(f))
\]
is the map obtained by applying $\Hom_R(-,R/(f))$ to $F^e_*R/(f)\xrightarrow{\cdot F^e_*f^{p^{e}-1}}F^e_*R/(f^{p^e})$.

First suppose that $R$ is Gorenstein and consider the short exact sequence 
\[
0\to R\xrightarrow{\cdot f} R\to R/(f)\to 0.
\]
Then $\Ext^1_R(F^e_*R,R)=0$ since $F^e_*R$ is Cohen-Macaulay and $R\cong R(K_X)$ is the canonical module of $R$, \cite[Theorem~3.3.10]{BH93}. Therefore the natural map $\Psi_1:\Hom_R(F^e_*R,R)\to \Hom_{R/(f)}(F^e_*R/(f),R/(f))$ is indeed onto under the Gorenstein hypothesis. To see that $\Psi_2$ is onto under the Gorenstein hypothesis we begin by considering the short exact sequence
\[
0\to F^e_*R/(f)\xrightarrow {\cdot F^e_*f^{p^e-1}}F^e_*R/(f^{p^e})\to F^e_*R/(f^{p^e-1})\to 0.
\]
Then $\Ext^1_{R/(f)}(F^e_*R/(f^{p^e-1}), R/(f))=0$ since $R/(f)$ is Gorenstein and therefore the map $\Psi_2$ is onto and it follows that $\Psi=\Psi_2\circ \Psi_1$ is the composition of two onto maps, provided $R$ is Gorenstein.

Now we show $\Hom_R(F^e_*R,R)\to \Hom_{R/(f)}(F^e_*R/(f), R/(f))$ is onto under the milder hypothesis that $R$ is $\Q$-Gorenstein of index relatively prime to $p$. Recall that we choose $e$ large enough so that $(1-p^e)K_X\sim 0$. Using the fact that $R$ is $(G_1)$ and reflexifying, we find that the module $\Hom_R(F^e_*R,R)$ can be identified with $F^e_*R$ via
\begin{align*}
    \Hom_R(F^e_*R,R)&\cong \Hom_R(F^e_*R\otimes R(K_X), R(K_X))\cong \Hom_R(F^e_*R(p^eK_X), R(K_X))\\ &\cong F^e_*\Hom_R(R(p^eK_X), R(K_X))\cong F^e_*R((1-p^e)K_X)\cong F^e_*R.
\end{align*}
The ring $R/(f)$ is $\Q$-Gorenstein of index which divides the index of $R$ by Lemma~\ref{Proposition Q-Gorenstein index under base change}. Therefore it is also the case that $\Hom_{R/(f)}(F^e_*R/(f),R/(f))\cong F^e_*R/(f)$. In particular, when viewed as an $F^e_*R/(f)$-module, the image of $\Hom_R(F^e_*R,R)\to \Hom_{R/(f)}(F^e_*R/(f),R/(f))$ is cyclic and therefore $(S_2)$. We can then check that the image agrees with with entire module at the codimension $1$ points of $\Spec(R/(f))$ by \cite[Theorem~1.12]{Har94}. We are assuming $R/(f)$ is $(G_1)$ and so the surjectivity of the desired map follows by the surjectivity of the map in the Gorenstein scenario.
\end{proof}

Suppose that $(R,\fm,k)$ is a $(G_1)$ and $(S_2)$ local ring. If $D$ is a divisor and $f\in \fm$ a non-zero-divisor of $R$ such that $R/(f)$ is $(G_1)$ and $(S_2)$ and $R(D)/fR(D)$ is principal codimension $1$ $R/(f)$-module, then we let $D|_{V(f)}$ denote a choice of divisor of $R/(f)$ so that $R/(f)(D|_{V(f)})$ is isomorphic to the reflexification $(R(D)/fR(D))^{un}:=\Hom_{R/(f)}(\Hom_{R/(f)}(R(D)/fR(D), R/(f)),R/(f))$ of $R(D)/fR(D)$. The following lemma provides a criterion for when the cyclic cover of a torsion divisor base changes to a cyclic cover of $R/(f)$.

\begin{lemma}
\label{Lemma p torsion cyclic cover and deformation}
Let $(R,\fm,k)$ be a $(G_1)$ and $(S_2)$ local $F$-finite ring of prime characteristic $p>0$. Suppose that $D$ is a torsion divisor of index $N$ and $R(ND)=R\cdot u$. Suppose that $f\in \fm$ is a non-zero-divisor such that $R/(f)$ is $(G_1)$ and $(S_2)$. For each $1\leq i\leq N-1$ suppose that $R(-iD)/fR(-iD)$ is an $(S_2)$ $R/(f)$-module which is principal in codimension $1$. If $S=\bigoplus_{i=0}^\infty R(-iD)t^{-i}/(u^{-1}t^{-N}-1)$ is the cyclic cover of $R$ associated with $D$ then $S/fS$ is the cyclic cover of $R/(f)$ with respect to $D|_{V(f)}$.
\end{lemma}

\begin{proof}
We must verify two things:
\begin{enumerate}
\item for each $1\leq i\leq N-1$ the ideal $R(-iD)/fR(-iD)$ is isomorphic to $R/(f)(-i D|_{V(f)})$;\label{lemma-cyclic-1}
\item the index of $D|_{V(f)}$ is $N$.\label{lemma-cyclic-2}
\end{enumerate} 
Our assumptions allow us to check (\ref{lemma-cyclic-1}) at the height $1$ primes of $\Spec(R/(f))$. By assumption, each $R(-iD)/fR(-iD)$ is principal at codimension $1$ points of $\Spec(R/(f))$. Hence $R(-iD)$ is principal at height $2$ points of $\Spec(R)$ containing $f$ and so $R(-iD)$ agrees with the $i$th ordinary power ideal $R(-D)^i$ at such points of $\Spec(R)$. Therefore $R(-iD)/fR(-iD)$ is indeed isomorphic to $R/(f)(-i D|_{V(f)})$ as claimed.

For (\ref{lemma-cyclic-2}) we first notice that $R(-ND)\cong R$ and so $R(-ND)/fR(-ND)$ is an $(S_2)$ and principal in codimension $1$ module of $R/(f)$. Therefore by the above, $R(-ND)/fR(-ND)=R/(f)(-ND_{V(f)})\cong R/(f)$ and the index of $D|_{V(f)}$ cannot exceed the index of $D$. However, if there was an $1\leq i\leq N-1$ so that $-iD|_{V(f)}\sim 0$ then $R(-iD|_{V(f)})\cong R(-iD)/fR(-iD)\cong R/(f)$. By Nakayama's Lemma the module $R(-iD)$ is a principal module of $R$ and the index of $D$ would be strictly less than $N$, a contradiction to our initial assumptions.
\end{proof}

If $(R,\fm,k)$ is a local strongly $F$-regular ring and $D$ is a torsion divisor, then $R(D)$ is a direct summand of $F^e_*R$ for some $e\in \mathbb{N}$, see \cite[Proof of Proposition~2.6]{Mar22}. It will likely not be the case that every torsion divisorial ideal in an $F$-pure local ring is a direct summand of $F^e_*R$ for some $e\in \mathbb{N}$, but the following lemma points out that the divisorial ideals of index $p$ to a power are a direct summand of $F^e_*R$ for some $e\in \mathbb{N}$, cf. Lemma~\ref{Lemma peth symbolic power of canonical is a direct summand} below.

\begin{lemma}
\label{Lemma p torsion divisors are direct summands}
Let $(R,\fm,k)$ be a local $(G_1)$ and $(S_2)$ $F$-finite and $F$-pure ring of prime characteristic $p>0$. If $D$ is a torsion divisor of index $p^e$ then $R(D)$ is a direct summand of $F^e_*R$.
\end{lemma}

\begin{proof}
The $e$th iterate of the Frobenius map $R\to F^e_*R$ splits as an $R$-linear map. If we tensor with $R(D)$ and reflexify we find that $R(D)$ is a direct summand of $F^e_*R(p^eD)\cong F^e_*R$.
\end{proof}

\begin{theorem}
\label{Theorem F-purity deforms}
Let $(R,\fm,k)$ be a local $F$-finite ring of prime characteristic $p>0$. Suppose that $R$ is $\Q$-Gorenstein and $f\in \fm$ is a non-zero-divisor such that $R/(f)$ is $(G_1)$, $(S_2)$, and $F$-pure. Then $R$ is $F$-pure.
\end{theorem}

\begin{proof} If $\dim(R)\leq 2$ then $R/(f)$ being $(G_1)$ implies that $R$ is Gorenstein. The property of being $F$-pure is equivalent to being $F$-injective in Gorenstein rings and $F$-injectivity is known to deform in Cohen-Macaulay rings, see \cite[Lemma~3.3 and Theorem~3.4]{Fed83}. Thus we may assume that $\dim(R)\geq 3$, and by induction we may assume that $R$ is $F$-pure at every non-maximal prime of $\Spec(R)$ containing the element $f$. Let $K_X$ be a choice of canonical divisor of $X=\Spec(R)$. Suppose that $K_X$ has index $np^e$ where $n$ is relatively prime to $p$. The divisor $D=nK_X$ has index $p^e$ and we suppose that $R(p^eD)=R\cdot u$. Let $S=\bigoplus_{i=0}^{\infty}R(-iD)t^{-i}/(u^{-1}t^{-p^e}-1)$ be the cyclic cover of $R$ corresponding to the divisor $D$. The work that follows will allow us to utilize Lemma~\ref{Lemma p torsion cyclic cover and deformation} and establish that $S/fS$ is a cyclic cover of $R/(f)$. 

 If $\fp\in \Spec(R)\setminus\{\fm\}$ and $f\in \fp$ then $R(-iD)_\fp$ is a direct summand of $F^e_*R_\fp$ by Lemma~\ref{Lemma p torsion divisors are direct summands}. Hence $\depth(R(-iD)_\fp)\geq \min\{\Height (\fp),3\}$ for all primes $\fp\in\Spec(R)\setminus\{\fm\}$ containing $f$. Therefore the quotients $R(-iD)/fR(-iD)$ are $(S_2)$ $R/(f)$-modules on the punctured spectrum of $R/(f)$. In particular, if $C_i$ denotes the cokernel of 
\[
R(-iD)/fR(-iD)\subseteq \left(R(-iD)/fR(-iD)\right)^{un}\cong R/(f)(-iD|_{V(f)})
\]
then $C_i$ is a finite length $R$-module and $H^i_\fm(R(-iD)/fR(-iD))\cong H^i_\fm\left((R/(f))(-iD|_{V(f)})\right)$ for all $i\geq 2$. We aim to show that  $R(-iD)/fR(-iD)$ is an $(S_2)$ $R/(f)$-module for each $1\leq i\leq p^e-1$. The work above reduces this problem to showing $\depth(R(-iD)/fR(-iD))\geq 2$. 

Consider the following commutative diagram whose horizontal arrows are $R$-linear and whose vertical arrows in the top square are $p^e$-linear:\footnote{By definition, a $p^e$-linear map of $R$-modules $N\to M$ is the same as an $R$-linear map $N\to F^e_*M$.}

\begin{equation*}
    \begin{tikzcd} \displaystyle
    \frac{R(-iD)}{fR(-iD)}\arrow[r] \arrow[d,"F^e"] & \displaystyle \frac{R}{(f)}\left(-iD|_{V(f)}\right)\arrow[d,"F^e"]\\ \displaystyle
    \frac{R(-ip^eD)}{fR(-ip^eD)}\arrow[r]\arrow[d,"\cong"] & \displaystyle \frac{R}{(f)}\left(-ip^eD|_{V(f)}\right)\arrow[d,"\cong"]\\ \displaystyle
    \frac{R}{(f)}\arrow[r] & \displaystyle \frac{R}{(f)}
    \end{tikzcd}
\end{equation*}
Recall that $R(p^eD)\cong R$, so by Lemma~\ref{Lemma p torsion cyclic cover and deformation} the middle horizontal map of the above commutative diagram is an isomorphism. By Lemma~\ref{Lemma p torsion divisors are direct summands} the right most vertical map is a split map. It follows that for all $i\geq 2$ the $p^e$-linear maps of local cohomology modules
\[
H^i_\fm(R(-iD)/fR(-iD))\to H^i_\fm(R(-ip^eD)/fR(-ip^eD))
\]
is a split map of abelian groups. In particular, the above $p^e$-linear maps on local cohomology modules are injective.

Now we consider the following commutative diagram:
\begin{equation*}
    \begin{tikzcd}
    0 \arrow[r]& R(-iD) \arrow[r,"\cdot f"]& R(-iD) \arrow[r]\arrow[d,"F^e"] & \displaystyle \frac{R(-iD)}{fR(-iD)}\arrow[r]\arrow[d,"F^e"] & 0 \\ 
    \, & \, & R(-ip^eD) \arrow[r]\arrow[d,"\cong"] & \displaystyle \frac{R(-ip^eD)}{fR(-ip^eD)}\arrow[d,"\cong"] & \, \\ 
    \, & \, & R \arrow[r] & \displaystyle \frac{R}{(f)} & \,
    \end{tikzcd}
\end{equation*}
The top row of the above diagram is a short exact sequence of $R$-modules and the composition of the vertical maps are $p^e$-linear. There is an induced commutative diagram of local cohomology modules whose top row is exact:
\begin{equation*}
    \begin{tikzcd}
     0 \arrow[r] & \displaystyle H^1_\fm\left(\frac{R(-iD)}{fR(-iD)}\right)\arrow[r] & H^2_\fm(R(-iD))\arrow[r,"\cdot f"] & H^2_\fm(R(-iD))\arrow[r,"\pi"]\arrow[d,"F^e"] & \displaystyle H^2_\fm\left(\frac{R(-iD)}{fR(-iD)}\right) \arrow[d,"F^e"] \\
     \, & \, & \, & H^2_\fm(R) \arrow[r]\arrow[d,"\cong"] & \displaystyle H^2_\fm\left(\frac{R}{(f)}\right) \\
     \, & \, & \, & 0
    \end{tikzcd}
\end{equation*}
We first remark that $H^2_\fm(R)$ is indeed the $0$-module. We are assuming $R/(f)$ is $(S_2)$ and $f$ is a non-zero-divisor, hence $R$ has depth at least $3$ and $H^2_\fm(R)=0$. The right most $p^e$-linear map is injective. Therefore the map $\pi$ is the $0$-map and $H^2_\fm(R(-iD))=fH^2_\fm(R(-iD))$. Because $R$ is excellent and $R(-iD)$ is $(S_2)$, the completed module $\widehat{R(-iD)}$ is an $(S_2)$ $\widehat{R}$-module and therefore $H^2_\fm(R(-iD))$ has finite length.\footnote{If $M\otimes_R \hat{R}$ is an $(S_2)$ $\hat{R}$-module then the Matlis dual of $H^2_\fm(M)$ is $\widehat{\Ext^{d-1}_R(M,R(K_X))}$. If we localize at a non-maximal prime $\fp$ then the $R_\fp$-Matlis dual of $\widehat{\Ext^{d-1}_R(M,R(K_X))}_\fp$ is $H^{\Height(\fp)-(d-1)}_\fm(M_\fp)=0$ as $\Height(\fp)-(d-1)\leq 1$ and $\widehat{M}_\fp$ has depth at least $2$. Therefore the finitely generated module $\Ext^{d-1}_R(M,R(K_X))$ is Artinian and so its Matlis dual, $H^2_\fm(M)$, is Noetherian.} By Nakayama's Lemma we have that $H^2_\fm(R(-iD))=0$, therefore $H^1_\fm(R(-iD)/fR(-iD))=0$, and the module $R(-iD)/fR(-iD)$ is an $(S_2)$ $R/(f)$-module for each $1\leq i\leq p^e-1$.

Suppose that $R(p^eD)=R\cdot u$ and let $S=\bigoplus_{i=0}^\infty R(-iD)t^{-i}/(u^{-1}t^{-p^e}-1)$ be the cyclic cover of $R$ corresponding to $D$. The ring $R$ is $F$-pure if and only if $S$ is $F$-pure by \cite[Proposition~4.20]{Car22}. Moreover, the ring $S$ is $\Q$-Gorenstein of index $n$, a number relatively prime to $p$, see Lemma~\ref{Lemma cyclic cover changing index}. Thus to verify $S$ is $F$-pure it suffices to check $S/fS$ is $F$-pure by Proposition~\ref{Proposition F-purity deforms if index not divisible by p}. We are assuming that $R/(f)$ is $F$-pure and the work above allows us to utilize Lemma~\ref{Lemma p torsion cyclic cover and deformation} and claim that $S/fS$ is the cyclic cover of $R/(f)$ corresponding to the divisor $D|_{V(f)}$. Therefore $S/fS$ is $F$-pure by a second application of \cite[Proposition~4.20]{Car22} and we conclude that $R$ is $F$-pure.
\end{proof}

\subsection{\texorpdfstring{$F$}{F}-purity is \texorpdfstring{$\fm$}{m}-adically stable in Cohen-Macaulay \texorpdfstring{$\Q$}{ℚ}-Gorenstein rings}
\label{Section F-pure}

Our strategy to show $F$-purity is $\fm$-adically stable in a Cohen-Macaulay $\Q$-Gorenstein ring $(R,\fm,k)$ is different than our strategy for deformation. We study the Frobenius degeneracy ideals of $R/(f+\epsilon)$ and compare them to the Frobenius degeneracy ideals of the quotient $R/(f)$. In light of Proposition~\ref{Proposition degeneracy ideals as colon ideals} and Lemma~\ref{Lemma colon ideal in G_1 rings} it is advantageous for there to be some natural number $e$ so that if $J\subsetneq R$ is a canonical ideal of $R$ then $J^{(p^e)}$ is a Cohen-Macaulay $R$-module.  

\begin{lemma}
\label{Lemma peth symbolic power of canonical is a direct summand}
Let $(R,\fm,k)$ be a $\Q$-Gorenstein $F$-pure local ring of prime characteristic $p>0$. Then there exists integers $e_0,e\in \mathbb{N}$ with $e\geq 1$  so that $R(p^eK_X)$ is a direct summand of $F^{e_0}_*R(K_X)$. In particular, if $R$ is Cohen-Macaulay then there exists an $e\geq 1$ so that $R(p^eK_X)$ is Cohen-Macaulay.
\end{lemma}

\begin{proof}
 Since $R$ is $F$-pure we have that for every integer $e\geq 1$ that $F^{e}_*R\cong R\oplus M$ has a free summand. Applying $\Hom_R(-,R(K_X))$ we find that $\Hom_R(F^{e}_*R,R(K_X))\cong F^{e}_*R(K_X)$ and $\Hom_R(R\oplus M, R(K_X))\cong  R(K_X)\oplus \Hom_R(M,R(K_R))$. Therefore $\Hom_R(F^{e}_*R,R(K_X))$ has an $R(K_X)$-summand for every $e\geq 1$. Suppose that the $\Q$-Gorenstein index of $R$ is $np^{e_0}$ where $p$ does not divide $n$. We choose $e\geq 1$ so that $n$ divides $p^{e}-1$. If we apply $-\otimes_R R((p^{e}-1)K_X)$ and then reflexify the direct sum decomposition $F^{e_0}_*R(K_X)\cong R(K_X)\oplus M'$ we find that
\[
F^{e_0}_*R(K_X+p^{e_0}(p^{e}-1)K_X)\cong F^{e_0}_*R(K_X)\cong R(p^{e}K_X)\oplus M''.
\]
\end{proof}

\begin{theorem}
\label{Theorem F-purity deforms Q-Gorenstein index p^e_0}
Let $(R,\fm,k)$ be a local $F$-finite $\Q$-Gorenstein Cohen-Macaulay ring of prime characteristic $p>0$. Suppose that $f\in \fm$ is a non-zero-divisor and $R/(f)$ is $(G_1)$ and $F$-pure. Then $R/(f+\epsilon)$ is $F$-pure for all $\epsilon\in\fm^{N\gg 0}$.
\end{theorem}

\begin{proof}
The ring $R$ is $F$-pure by Theorem~\ref{Theorem F-purity deforms}. By Lemma~\ref{Lemma peth symbolic power of canonical is a direct summand} there exists an $e\in\mathbb{N}$ so that $R(p^eK_X)$ is a Cohen-Macaulay $R$-module.

We aim to show that the $e$th Frobenius degeneracy ideals $I_e(R/(f+\epsilon))$ of $R/(f+\epsilon)$ are proper ideals for all $\epsilon\in\fm^{N\gg 0}$. If the dimension of $R$ is no more than $2$ then $R/(f)$ being $(G_1)$ implies that $R$ is Gorenstein. In the Gorenstein setting, $F$-purity and $F$-injectivity are equivalent, and the latter is known to deform and be $\fm$-adically stable in Cohen-Macaulay rings, see \cite[Lemma~3.3 and Theorem~3.4]{Fed83} and \cite[Corollary 4.9]{DSS20}. Thus we may assume $R$ is of dimension $d+1\geq 3$.

We may select canonical ideal $J\subsetneq R$ so that $J/fJ\cong (J,f)/(f)$ is the canonical ideal of $R/(f)$, see Lemma~\ref{lemma base change the canonical module}. By Proposition~\ref{Proposition Gorenstein in codimension 1 is stable} we know that the rings $R$ and $R/(f+\epsilon)$ are $(G_1)$ for all $\epsilon\in\fm^{N\gg 0}$. Furthermore, the rings $R/(f+\epsilon)$ are $\Q$-Gorenstein of index dividing the $\Q$-Gorenstein index of $R$ for all $\epsilon\in\fm^{N\gg0 }$, see Proposition~\ref{Proposition Q-Gorenstein index under base change}. Also note that for all $\epsilon\in \fm^{N\gg0}$ we have that $f+\epsilon$ avoids all components of $J$ by the Krull intersection theorem (see also the proof of Proposition \ref{Proposition Gorenstein in codimension 1 is stable}). Therefore $J^{(n)}/(f+\epsilon)J^{(n)}\cong (J^{(n)}, f+\epsilon)/(f+\epsilon)$.

Since $J^{(p^e)}$ is Cohen-Macaulay, it follows that for all $\epsilon\in \fm^{N\gg 0}$ that $(J^{(p^e)}, f+\epsilon)/(f+\epsilon)$ is a Cohen-Macaulay $R/(f+\epsilon)$-module. By Lemma~\ref{Lemma base change of divisorial ideal} $(J^{(p^e)}, f+\epsilon)/(f+\epsilon)$ is an unmixed ideal of height $1$ and 
\[
\frac{(J^{(p^e)}, f+\epsilon)}{(f+\epsilon)}=\left(\frac{(J,f+\epsilon)}{(f+\epsilon)}\right)^{(p^e)}.
\]

To ease notation we write $R_\epsilon$ to denote the quotient $R/(f+\epsilon)$. The above tells us that $J^{(p^e)}R_\epsilon=(JR_\epsilon)^{(p^e)}$ for all $\epsilon\in\fm^{N\gg 0}$. Choose a non-zero-divisor $x_1\in J$ on $R$ and $R/(f)$ and let $W$ denote the multiplicatively closed set given by the complement of the union of the minimal primes of the unmixed ideal $(x_1,f)$. Since $R/(f)$ is $(G_1)$ the localized ideal $\left((J,f)/(f)\right)R_W$ is principal. Hence there exists a parameter element $x_2$ on $R/(x_1,f)$ and $a\in J$ so that $x_2J\subseteq (a)\subseteq J$.

We extend $x_2$ to a full parameter sequence $x_2,x_3,\ldots ,x_d$ on $R/(x_1,f)$ and choose a socle generator $u\in R$ on $R/(J,x_2,\ldots,x_d,f)$. Observe that for all $t\gg 0$ and all $\epsilon\in\fm^{N\gg 0}$,
\begin{align}
    I_e(R_\epsilon)=(x_1^{t-1}JR_\epsilon,x_2^t,\ldots,x_d^t)^{[p^e]}:_{R_\epsilon}(x_1\cdots x_d)^{(t-1)p^e}u^{p^e}\label{thm-stability-1}\\
    =(JR_\epsilon,x_2^t,\ldots,x_d^t)^{[p^e]}:_{R_\epsilon}(x_2\cdots x_d)^{(t-1)p^e}u^{p^e}\label{thm-stability-2}\\
    =((JR_\epsilon)^{(p^e)},x_2^{2p^e},x_3^{tp^e}\ldots,x_d^{tp^e}):_{R_\epsilon}(x_3\cdots x_d)^{(t-1)p^e}(x_2u)^{p^e}\label{thm-stability-3}\\
    =((JR_\epsilon)^{(p^e)},x_2^{p^e},x_3^{p^e}\ldots,x_d^{p^e}):_{R_\epsilon}u^{p^e}\label{thm-stability-4}\\
    =(JR_\epsilon,x^2_2,x_3\ldots,x_d)^{[p^e]}:_{R_\epsilon}(x_2u)^{p^e}\label{thm-stability-5}
\end{align}
where (\ref{thm-stability-1}) follows from Proposition~\ref{Proposition degeneracy ideals as colon ideals}, (\ref{thm-stability-2}) follows from Lemma~\ref{Lemma Removing the x1}, and both (\ref{thm-stability-3}) and (\ref{thm-stability-5}) follow from Lemma~\ref{Lemma colon ideal in G_1 rings}. To see (\ref{thm-stability-4}), notice that since $J^{(p^e)}$ is a Cohen-Macaulay $R$-module, and consequently $(JR_\epsilon)^{(p^e)}$ are Cohen-Macaulay $R_\epsilon$-modules, the quotient $R/J^{(p^e)}$ is Cohen-Macaulay of dimension $d$ and the quotients $R_\epsilon/(JR_\epsilon)^{(p^e)}$ are Cohen-Macaulay of dimension $d-1$. 

Observe now that for all $\epsilon\in\fm^{N\gg 0}$
\[
I_e(R_\epsilon)=\frac{((J,x^2_2,x_3\ldots,x_d)^{[p^e]},f+\epsilon):_{R}(x_2u)^{p^e}}{(f+\epsilon)}.
\]
Thus $R/(f)$ is $F$-pure if and only if the following colon ideal is a proper ideal of $R$:
\[
((J,x^2_2,x_3\ldots,x_d)^{[p^e]},f):_{R}(x_2u)^{p^e}.
\]
Observe that if $\epsilon\in (J,x^2_2,x_3\ldots,x_d,f)$ then 
\[
((J,x^2_2,x_3\ldots,x_d)^{[p^e]},f+\epsilon):_{R}(x_2u)^{p^e}=((J,x^2_2,x_3\ldots,x_d)^{[p^e]},f):_{R}(x_2u)^{p^e}.
\]
Therefore $I_e(R_\epsilon)$ is a proper ideal of $R_\epsilon$ for all $\epsilon\in\fm^{N\gg 0}$, i.e. $R_\epsilon$ is $F$-pure and the property of being $F$-pure is indeed $\fm$-adically stable.
\end{proof}

\section*{Acknowledgements}

We thank Ian Aberbach, Linquan Ma, Takumi Murayama, and Karl Schwede for useful conversations during the preparation of this article. We also thank Pham Hung Quy for pointing out an inaccuracy in a previous draft. We thank Linquan Ma, Nicholas Cox-Steib, and Kazuma Shimomoto for comments on a previous draft. The second named author is grateful to his advisor, Kevin Tucker, for his constant encouragement. Finally, we thank the anonymous referees for carefully reading an earlier version of this article and for suggesting many helpful improvements.

\newpage

\begin{table}[h!]
\tiny
\centering
\caption{Summary of deformation and $\fm$-adic stability of $F$-singularities for $F$-finite local rings $(R,\fm)$}\label{summary-table}
\begin{threeparttable}
\begin{tabular}{c|cc }
  $\cP$ & $\cP$ deforms? & $\cP$ is $\fm$-adically stable? \\
  \hline
  $F$-injective& \textbf{open}; Yes\tnotex{t:fi}\hspace{.1cm} \cite[Thm. 3.4(1)]{Fed83}\tnotex{t:cm} & Yes \cite[Cor. 4.9]{DSS20}\tnotex{t:cm}\\
  $F$-pure & No in general \cite{Fed83}; Yes (Theorem \ref{Theorem F-purity deforms})\tnotex{t:qg}&No in general \cite[Thm. 5.3]{DSS20}; Yes (Theorem \ref{Theorem F-purity deforms Q-Gorenstein index p^e_0})\tnotex{t:qg}\hspace{.45cm}\textsuperscript{,}\tnotex{t:cm}\\
  $F$-rational & Yes \cite[Thm. 4.2(h)]{HH94}& Yes \cite[Cor. 3.9]{DSS20}\\
  Strongly $F$-regular&No in general \cite[Thm. 1.1]{Sin99}; Yes \cite[3.3.2]{Mac96}\tnotex{t:a}\hspace{.1cm}+\cite{AKM98}\tnotex{t:qg}\hspace{.5cm} &  No in general \cite[5.3]{DSS20}; Yes \cite[Thm. 5.11]{DSS20}\tnotex{t:qgps}\hspace{.6cm}\textsuperscript{,}\\
  $F$-nilpotent &No \cite[Ex. 2.8(2)]{ST17}&No \cite[Ex. 2.8(2)]{ST17}+\cite[Thm. 2.4]{DSS20}+\cite[Thm. 5.5]{KMPS19}\\
  $F$-anti-nilpotent&Yes \cite[Thm. 4.2(i)]{MQ18}&\textbf{open}\\
  $F$-full &Yes \cite[Thm. 4.2(ii)]{MQ18}& \textbf{open}\\
  \hline
\end{tabular}
\begin{tablenotes}\tiny
\item[1] \label{t:fi} See also \cite[Theorem 3.7 \& Corollary 4.7]{HMS14}, \cite[Theorem 5.11]{MQ18} and \cite{DSM22}\\
\item[2] \label{t:a} See also \cite[2.2.4]{Abe02}\\
\item[$\Q$-Gor] \label{t:qg} $R$ is $\Q$-Gorenstein
\item[$\Q$-Gor$^\circ$] \label{t:qgps} $R$ is $\Q$-Gorenstein on the punctured spectrum
\item[C-M] \label{t:cm}$R$ is Cohen-Macaulay
\item[Gor] \label{t:g} $R$ is Gorenstein
\end{tablenotes}
\end{threeparttable}
\end{table}

\bibliographystyle{alpha}
\bibliography{References}

\end{document}